\documentclass[reqno,10pt]{amsart}
\usepackage{color}

\usepackage{mathtools}
\usepackage{amsfonts}
\usepackage{amssymb}
\usepackage{amsthm}
\usepackage{newlfont}
\usepackage{graphicx}
\usepackage{float}
\usepackage{amsmath}
\usepackage{palatino,mathpazo}
\usepackage{extarrows}
\usepackage{titleps}
\usepackage{booktabs}
\usepackage{geometry}
\usepackage{tikz}
\usetikzlibrary{matrix,shapes,arrows,positioning,chains}

\numberwithin{equation}{section}
\newtheorem{theorem}{Theorem}[section]
\newtheorem{lemma}{Lemma}[section]
\newtheorem{remark}{Remark}[section]

\newcommand{\ed}{\end {document}}

\newcommand{\R}{\mathbb{R}}

\begin{document}
\title[Helmholtz equation with arbitrary order]{\bf On the equivalence of classical Helmholtz equation and fractional Helmholtz equation with arbitrary order}

\author[X.Y. Cheng]{Xinyu Cheng}
\address{Xinyu Cheng, School of Mathematical Sciences, Fudan University, Shanghai, P.R. China.}
\email{xycheng@fudan.edu.cn}

\author[D. Li]{Dong Li}
\address{Dong Li, SUSTech International Center for Mathematics, and Department of Mathematics,  Southern University of Science and Technology, Shenzhen, P. R. China.}
\email{lid@sustech.edu.cn}

\author[W. Yang]{ Wen Yang}
\address{\noindent Wen ~Yang,~Wuhan Institute of Physics and Mathematics, Innovation Academy for Precision Measurement Science and Technology, Chinese Academy of Sciences, Wuhan 430071, P. R. China.}
\email{math.yangwen@gmail.com}

\renewcommand{\thefootnote}{\fnsymbol{footnote}}
	
\maketitle
	
{\noindent\small{\bf Abstract:}
We show the equivalence of the classical Helmholtz equation and the fractional Helmholtz
equation with arbitrary order. This improves a recent result of Guan, Murugan and Wei \cite{gmw2022}.

}
	
\vspace{1ex}

\section{Introduction}
In mathematical physics the classical eigenvalue problem for the Laplace operator is known as the Helmholtz equation. In a
prototypical setup one is interested in finding eigen-pairs to  the linear partial differential equation
\begin{equation} \label{a1}
-\Delta f= \lambda f
\end{equation}
with various boundary conditions.  The Helmholtz equation corresponds to the time-independent form of the wave equation as it naturally appears in reducing the complexities of the solution procedure by the usual separation of variable method.  The Helmholtz equation is widely used in a plethora of the physical and engineering applications such as heat conduction, acoustic radiation, water wave propagation and some related applied science. If one replaces the Laplacian operator by
the fractional Laplacian (i.e.
$(-\Delta)^s$ for some $s>0$) on the left hand side  of \eqref{a1}, then one obtains a so-called fractional version of the Helmholtz equation which
 also plays an important role in physics. Recently, starting from the Maxwell's equations, the article \cite{wva2020} derived a scalar fractional Helmholtz equation. There exists a deep connection between solutions to the classical Helmholtz equation and the fractional ones. In recent \cite{gmw2022},  Guan, Murugan and Wei considered the fractional Helmholtz equation
 posed in the whole space and obtained some equivalence of the corresponding solution to the classical Helmholtz equation
 for the regime $0<s<1$ under some decay assumptions at spatial infinity.  The purpose of this note is to prove an
 optimal Liouville type result for all $s>0$ without any additional decay assumptions.

Before describing the main result we fix some notation used throughout this note.
Let $s>0$.  For $u\in \mathcal
S(\mathbb R^d)$, $d\ge 1$, the fractional Laplacian $\Lambda^s u
=(-\Delta)^{\frac s2} u $
is defined via Fourier transform as
\begin{align}
\widehat{\Lambda^s u } (\xi) = |\xi|^s \widehat{u}(\xi), \qquad \xi \in \mathbb R^d.
\end{align}
Here we adopt the following convention for Fourier transform:
\begin{align}
& (\mathcal F u)(\xi) =\widehat u (\xi) = \int_{\mathbb R^d} u(y) e^{-iy\cdot \xi} dy ; \\
& u(x) = \frac 1 {(2\pi)^d} \int_{\mathbb R^d} \widehat{u}(\xi) e^{i \xi \cdot x} d\xi
=: (\mathcal F^{-1} \widehat u )(x).
\end{align}
For $f_1:\, \mathbb R^d \to \mathbb C$, $f_2:\, \mathbb R^d \to \mathbb C$,  $f_1$, $f_2$ Schwartz, we denote
the usual $L^2$ pairing:
\begin{align}
\langle f_1, f_2\rangle : = \int_{\mathbb R^d} f_1(x) \overline{f_2(x) } dx,
\end{align}
where $\overline{z}$ denotes the usual complex conjugate of $z\in \mathbb C$.  The usual
Plancherel formula reads
\begin{align}
\langle \widehat f_1, \widehat f_2 \rangle = (2\pi)^d \langle f_1, f_2 \rangle.
\end{align}
If we denote $f_3 =\widehat f_2$, then $f_2 = \mathcal F^{-1} (f_3)$. Thus for
$f_1$, $f_3 \in \mathcal S(\mathbb R^d)$, it holds that
\begin{align}
\langle \widehat f_1, f_3 \rangle = (2\pi)^d \langle f_1, \mathcal F^{-1} (f_3) \rangle.
\end{align}
For $u \in L^{\infty}(\mathbb R^d)$, $s>0$, we can define $\Lambda^s u \in \mathcal
S^{\prime}(\mathbb R^d)$ via the formula
\begin{align} \label{1.8a}
\langle \widehat{\Lambda^s u},  \phi \rangle 
& = (2\pi)^d \langle u, \mathcal F^{-1}( |\cdot|^s \phi )\rangle, \qquad \forall\, \phi \in \mathcal S(\mathbb R^d).
\end{align}

The main result of this note is the following.

\begin{theorem}[Equivalence of fractional Helmholtz and classical Helmholtz] \label{thm1}
Let $s>0$. Assume $u \in L^{\infty}(\mathbb R^d)$ satisfies $\Lambda^s u = u$ in $\mathcal S^{\prime}(\mathbb R^d)$. Then\footnote{More precisely $u$ can be identified as a $C^{\infty}(\mathbb R^d)$ function in the
spirit of the usual real analysis.}
$u \in C^{\infty}(\mathbb R^d) \cap L^{\infty}(\mathbb R^d)$ and $ -\Delta u =u $.  In particular for dimension $d=1$, we have
$u(x) = c_1 \cos x + c_2 \sin x$ for some constants $c_1$, $c_2$.
\end{theorem}
In recent \cite{gmw2022}, Guan, Murugan and Wei proved the following results:
\begin{itemize}
\item If $0<s<2$, $d=1$, $u\in L^{\infty}(\mathbb R)$ satisfies $\Lambda^s u = u$,  then $u(x) =c_1
\cos x + c_2 \sin x$.
\item If $0<s~{\le}~2$, $d\ge 2$, $u \in C^{\infty}(\mathbb R^d) \cap L^{\infty}(\mathbb R^d)$ satisfies
$\Lambda^s u = u$ and $\lim\limits_{|x| \to \infty} u (x) =0$, then $-\Delta u = u$.
\item If $m\in\mathbb{N}$, $d\ge 2$, $u \in C^{\infty}(\mathbb R^d) \cap L^{\infty}(\mathbb R^d)$ satisfies
$(-\Delta)^m u = u$ if and only if $-\Delta u = u$.
\end{itemize}
The proof of \cite{gmw2022} relies on an extension formula of the fractional Laplacian operator.
Our theorem \ref{thm1} gives a unifying treatment for all $s >0$. Note that a somewhat pleasing feature of our proof is that we do not need to impose the extra decay assumption of $u$ at spatial infinity. To put things into perspective, we mention that the case  $d=1$ was first obtained by Fall and Weth in \cite{fw2016}. In the past decade there appears a rather extensive literature on the topic of fractional elliptic equations (cf. \cite{cs2015,cco2006,fw2016,fls2013} and the references therein).

{\begin{remark}
It is possible to classify further. For example consider $u \in L^{\infty}(\mathbb R^d)
\cap C^{\infty}(\mathbb R^d)$ solving the Helmholtz equation in dimension $d=2$:
\begin{align}
0=\Delta u + u= \partial_{rr} u + \frac 1 r \partial_r u +\frac 1 {r^2} \partial_{\theta\theta} u
+ u.
\end{align}
Denote $h_n(r) = \int_{0}^{2\pi} u(r, \theta) \cos n \theta d\theta$ (or
$h_n(r)=\int_0^{2\pi} u(r,\theta) \sin n \theta d\theta$) where $n\ge 0$ is an integer.
Clearly $h_n$ solves the equation
\begin{align}
r^2 \partial_{rr} h_n + r \partial_r h_n + (r^2-n^2) h_n=0.
\end{align}
The general solution is $h_n(r) = c_1 J_n(r) +c_2 Y_n(r)$ where $J_n$ and $Y_n$ are the standard
Bessel functions. Since $h_n$ is regular near $r=0$, we deduce $h_n(r) =c_1 J_n(r)$. Thus we have
\begin{align}
u = \sum_{n=0}^{\infty} J_n(r) ( a_n \cos n \theta+ b_n \sin n \theta), \qquad
\text{in $\mathcal S^{\prime}(\mathbb R^2)$}.
\end{align}
More precisely, for any $\phi \in \mathcal S(\mathbb R^2)$, we have
\begin{align}
\lim_{N\to \infty} \langle u - \sum_{n=0}^N J_n(r) (a_n \cos n\theta+b_n
\sin n \theta), \, \phi \rangle =0.
\end{align}
Similar statements also hold for dimensions $d\ge 3$. 
\end{remark}
}

\begin{remark}
One should note that for $u\in L^{\infty}(\mathbb R^d)$, we use the definition of $\Lambda^s
=(-\Delta)^{\frac s2}$ via
the formula:
\begin{align}
\langle \Lambda^s u, \phi \rangle = \langle u,  \Lambda^s \phi \rangle, \quad \forall\,
\phi \in \mathcal S(\mathbb R^d).
\end{align}
This formula is equivalent to \eqref{1.8a}.  Thanks to this characterization, it is pedestrian to show that
this definition of $\Lambda^s u$ coincides with the usual Molchanov-Ostrovskii  extension formula
\cite{mo1969}
(see also Muckenhoupt-Stein \cite{ms1965}).
For example, we consider the general formula
\begin{align}
\langle \Lambda^s_{\text{new}} f, \phi \rangle=
\lim_{\varepsilon \to 0} \langle L_{\varepsilon} f , \phi \rangle, \quad \forall\, \phi \in \mathcal S(\mathbb R^d);
\end{align}
where $L_{\varepsilon}$ is the extension/regularization operator for $\varepsilon>0$. In the case of
Molchanov-Ostrovskii extension or more general higher order extension, $L_{\varepsilon}$ is a benign
linear operator.  Since $f \in L^{\infty}(\mathbb R^d)$, for each $\varepsilon>0$ we have
\begin{align}
	\langle L_{\varepsilon} f, \phi \rangle = \langle f, L_{\varepsilon} \phi \rangle.
\end{align}
Thus as long as $L_{\varepsilon} \phi \to \Lambda^s \phi$ in $L^1(\mathbb{R}^d)$, we obtain
\begin{align}
	\Lambda^s_{\text{new} } f = \Lambda^s f.
\end{align}
\end{remark}



\subsection*{Notation} For any two quantities $X$ and $Y$, we write $X\lesssim Y$ if $X\le CY$ for some
harmless constant $C>0$.

\section{Proof of Theorem \ref{thm1}}
 We first introduce the following $d+1$ functions such that
\begin{equation}
\label{2.1}
\begin{cases}
\chi_0(x)\quad\mbox{is supported in}\quad \{|x|\le 10\},\\
\chi_j(x)\quad\mbox{is supported in}\quad \{x\in\mathbb{R}^d:~|x|\ge9,~|x_j|\ge \frac{1}{2\sqrt{d}}|x|\},\qquad j=1,\cdots,d,
\end{cases}
\end{equation}
and
\begin{equation}
\label{2.2}
\sum_{j=0}^d\chi_j(x)=1.
\end{equation}	
	
Next {for $u\in L^{\infty}(\mathbb R^d)$, }we define
\begin{equation}
\label{2.3}
b_0(\xi)=\int_{\mathbb{R}^d}u(x)\chi_0(x)e^{-ix\cdot\xi}dx\quad\mbox{and}\quad
b_j(\xi)=\int_{\mathbb{R}^d}\frac{u(x)\chi_j(x)}{x_j^{d+1}}e^{-ix\cdot\xi}dx,~j=1,\cdots,d.
\end{equation}	

\begin{lemma}
\label{le2.1}
{Let $b_0$ and $b_j,~j=1,\cdots,d$ be defined in \eqref{2.3}, then
\begin{equation}
\label{2.4}
b_j\in L^2(\R^d) \cap C_b(\mathbb R^d),\quad j=0,\cdots,d,
\end{equation}
where $C_b$ denotes the set of continuous bounded functions on $\mathbb R^d$. }
In addition, for $j=1,\cdots,d$, we have
\begin{equation}
\label{2.5}
|b_j(\xi)-b_j(0)|\lesssim |\xi| |\log |\xi| |\quad \text{ for}\quad |\xi| \le \frac 1 2.
\end{equation}
\end{lemma}

\begin{proof}
Since $u\in L^\infty(\R^d)$, clearly
\begin{equation}
\label{2.6}
u(x)\chi_0(x)\in L^2(\R^d)\quad\mbox{and}\quad\frac{u(x)\chi_j(x)}{x_j^{d+1}}\in L^2(\R^d),~j=1,\cdots,d.
\end{equation}
For $|\xi|\le \frac 12$, we have
\begin{equation}
\label{2.7}
\begin{aligned}
|b_j(\xi)-b_j(0)|\lesssim ~& \int_{|x|\leq\frac{2}{|\xi|}}\frac{|x||\xi|}{|x|^{d+1}}dx
+\int_{|x|>\frac{2}{|\xi|}}\frac{1}{|x|^{d+1}}dx\\	
\lesssim~ & |\xi||\log|\xi||+|\xi|\lesssim |\xi||\log|\xi||.
\end{aligned}
\end{equation}
\end{proof}
Thanks to the cut-off functions $\chi_j(x),~j=0,\cdots,d$, we have the following decomposition.
\begin{lemma}
\label{le2.2}
Let the dimension $d\ge 1$.
Suppose $u \in L^{\infty}(\mathbb R^d)$. Then
\begin{equation*}
\widehat{u} (\xi)= b_0(\xi) + \sum_{j=1}^d (i \partial_{\xi_j} )^{d+1} b_j (\xi)
	\quad  \text{in}\quad \mathcal S^{\prime} (\mathbb R^d).
\end{equation*}
\end{lemma}
\begin{proof}
Obvious.
\end{proof}

Thanks to Lemmas \ref{le2.1} and \ref{le2.2}, for $u\in L^{\infty}(\mathbb R^d)$ and $s>0$, one can define
$\Lambda^s u \in \mathcal S^{\prime}(\mathbb R^d)$ via the formula
\begin{align}
\label{2.8}
\widehat{\Lambda^s u}(\xi) =|\xi|^s \widehat u(\xi) =|\xi|^s  b_0(\xi)+ \sum_{j=1}^d |\xi|^s (i \partial_{\xi_j} )^{d+1} b_j(\xi).
\end{align}
In particular one can check that the following pairing
\begin{align}
\label{2.9}
\left\langle |\xi|^s (i\partial_{\xi_j} )^{d+1} b_j(\xi), \phi(\xi) \right\rangle = \left\langle b_j(\xi) - b_j(0), (i\partial_{\xi_j})^{d+1}(  |\xi|^s \phi (\xi) ) \right\rangle,\quad \forall \phi \in \mathcal S(\mathbb R^d).
\end{align}
Here in the above the $L^2$-pairing $\langle \cdot , \cdot \rangle$ is for the variable $\xi$. We spell out the explicit
argument $\xi$ to indicate this dependence.

To prove theorem \ref{thm1}, it suffices to show the following theorem.

\begin{theorem} \label{t2.1a}
Let $s>0$. Suppose  $u \in L^{\infty}(\mathbb R^d)$ satisfies
\begin{align}
\label{2.10}
\left\langle u, ~\mathcal F^{-1}\Bigl(  ( |\xi|^s-1) \phi(\xi)   \Bigr) \right\rangle =0, \qquad \forall\, \phi
\in \mathcal S(\mathbb R^d).
\end{align}
Then
\begin{align}
\label{2.11}
\left\langle u, \mathcal F^{-1}\Bigl(  ( |\xi|^2-1) \psi(\xi)   \Bigr) \right\rangle =0, \qquad \forall\, \psi \in \mathcal S(\mathbb R^d).
\end{align}
Also we have for any $k>0$,
\begin{align}
\label{2.12a}
\left\langle u, \mathcal F^{-1}\Bigl(  (e^{-k( |\xi|^2-1)} -1)  \psi(\xi)   \Bigr) \right\rangle =0, \qquad \forall\, \psi \in \mathcal S(\mathbb R^d).
\end{align}
In particular $u =e^{k\Delta +k} u$ in $\mathbb S^{\prime}(\mathbb R^d)$ and in $L^{\infty}(\mathbb R^d)$.
It follows that $u$ can be identified as  a $C^{\infty}(\mathbb R^d)$ function.

Furthermore, the tempered distribution $\widehat u$ is compactly supported. More precisely, we have
\begin{align} \label{2.13a}
\mathrm{supp}(\widehat{u} ) \subset K= \{ \xi:\; |\xi| =1 \}.
\end{align}
\end{theorem}

\begin{proof}
The key is to localize to the regime $||\xi|-1| \ll 1$.  Choose $\chi_1 \in C_c^{\infty}(\mathbb R^d)$ such that
\begin{align}
\chi_1(\xi)=
\begin{cases}
1, \quad \text{$ |\xi| \le 1-\delta_0$};\\
0, \quad \text{$ |\xi| \ge 1-\frac {\delta_0} 2 $}.
\end{cases}
\end{align}
Similarly choose $\chi_2 \in C_c^{\infty}(\mathbb R^d)$ such that
\begin{align}
\chi_2(\xi)=
\begin{cases}
1, \quad \text{$ |\xi| \le 1+\frac 12\delta_0$};\\
0, \quad \text{$ |\xi| \ge 1+\delta_0 $}.
\end{cases}
\end{align}
In the above, the constant $\delta_0>0$ will be taken sufficiently small.
	
We first claim that
\begin{equation}
\label{2.14a}
\left\langle u, \mathcal F^{-1}\Bigl(   (|\xi|^2-1) \chi_1(\xi) \psi (\xi)   \Bigr) \right\rangle =0.
\end{equation}
Indeed, we set $\tilde \chi(\xi) = \chi_1(\xi) \psi (\xi)$.
Note that $\tilde \chi \in C_c^{\infty}(|\xi| < 2)$.
Let $\chi \in C_c^{\infty}(|z| <1)$ be such that $\chi(z)=1$ for $|z|\le \frac 12$ and $\chi(z)=0$ for
$|z| \ge \frac 23$.  By \eqref{2.10}, we have for any $0<\varepsilon\ll 1$,
\begin{align*}
\left\langle u , \mathcal F^{-1} (  (|\xi|^2-1) (1- \chi\left(\frac {\xi} {\varepsilon} \right) ) \tilde \chi (\xi) ) \right\rangle =0;
\end{align*}
and
\begin{align*}
\left\langle u , \mathcal F^{-1} (  (|\xi|^s- 1)  \chi\left(\frac {\xi} {\varepsilon} \right) \tilde \chi (\xi) ) \right\rangle =0. \notag
\end{align*}
Thus we only need to show
\begin{align*}
\lim_{\varepsilon \to 0}
\left\langle u , \mathcal F^{-1} (  (|\xi|^2- |\xi|^s)  \chi\left(\frac {\xi} {\varepsilon} \right) \tilde \chi (\xi) ) \right\rangle =0.
\end{align*}
This last assertion follows from Lemma \ref{le2.2}.

It is not difficult to check that
\begin{equation}
\label{2.15}
\left\langle u, \mathcal F^{-1} \Bigl(  (|\xi|^2-1) (1-\chi_1(\xi) ) (1-\chi_2( \xi))  \psi (\xi) \Bigr) \right\rangle =0.
\end{equation}
Thus it remains to show (below $\chi_3(\xi) = (1-\chi_1(\xi) ) \chi_2(\xi)$, note that it is localized to $||\xi|-1| \ll 1$)
\begin{align}
\label{2.16}
\left\langle u, \mathcal F^{-1} \Bigl ( (|\xi|^2-1) \chi_3(\xi) \psi (\xi) \Bigr) \right\rangle =0.
\end{align}
Write $\eta =|\xi|^s -1$. Note that $|\xi|^2 -1 = (1+\eta)^{\frac 2s}-1$.  By \eqref{2.10}, we clearly have
\begin{align}
\label{2.17}		
\left\langle  u,   \mathcal F^{-1} \Bigl(  \eta^\ell  \chi_3(\xi) \psi (\xi ) \Bigr) \right\rangle =0, \qquad \forall\, \ell\ge 1.
\end{align}
A crucial fact is used here: thanks to the cut-off $\chi_3(\xi)$, the function $\chi_3(\xi) \eta^{\ell-1} \psi (\xi)
\in \mathcal S(\mathbb R^d)$ for any $\ell\ge 1$.
	
Since $(1+\eta)^{\frac 2s} -1 = \sum\limits_{\ell\ge 1} c_\ell \eta^\ell$ (the expansion converges for $|\eta| \ll 1$),  it is not difficult
to check that
\begin{align}
\label{2.18}		
\lim_{N\to \infty} \sum_{\ell=1}^N \left(c_\ell \eta^\ell \chi_3(\xi) \psi (\xi) \right)= ((1+\eta)^{\frac 2s}-1)
\chi_3(\xi) \psi(\xi) \quad\text{in}\quad \mathcal S(\mathbb R^d).
\end{align}
Clearly \eqref{2.16} follows.

Next, the identity \eqref{2.12a} readily follows from \eqref{2.11}, since  for $\psi \in \mathcal S(\mathbb R^d)$
\begin{align}
\Bigl( 1- e^{-k(|\xi|^2-1)} ) \psi = (|\xi|^2-1) \underbrace{ \int_0^k e^{- \theta (|\xi|^2-1) } d\theta \psi (\xi) }_{\in \mathcal
S(\mathbb R^d)}.
\end{align}
Note that strictly speaking we should write
$\int_0^k e^{- \theta (|\cdot |^2-1) } d\theta \psi (\cdot) \in \mathcal S(\mathbb R^d)$ but we chose to spell
out the explicit argument $\xi$ for notational visibility. By \eqref{2.12a}, we have $u= e^{k\Delta +k } u$ for any $k>0$. Thus $u$ can be identified as a $C^{\infty}$
function thanks to the smoothing heat semi-group. {For example, one can take $k=1$ and note that
\begin{align}
(e^{\Delta+1} u)(x) = (\rho * u)(x),
\end{align}
where $\rho>0$ is a Schwartz function, and $*$ denotes the usual convolution. Since $u\in L^{\infty}(\mathbb R^d)$,
we clearly have $\rho*u \in C^{\infty}$. }

Finally we turn to \eqref{2.13a}. It suffices for us to show
\begin{align}
\langle \widehat u, \; \phi \rangle =0, \qquad \forall\, \phi \in C_c^{\infty}(\mathbb R^d \setminus K).
\end{align}
Since $\phi \in C_c^{\infty}( \mathbb R^d \setminus K)$, we have $\phi_1=\frac {\phi(\cdot) }{|\cdot|^2-1}
\in C_c^{\infty}(\mathbb R^d \setminus K)$.  Clearly
\begin{align}
\langle \widehat u, \; (|\cdot|^2-1) \phi_1 \rangle =0
\Rightarrow \langle \widehat u, \; \phi \rangle=0.
\end{align}
\end{proof}

{\center{\bf Acknowledgement}.} The research of the third author is supported by NSFC Grants 11871470 and 12171456.

\end{document}